\documentclass[reqno,11pt]{amsart}
\usepackage[dvipdfm]{hyperref}
\usepackage{latexsym}
\usepackage{amsfonts}
\usepackage{amssymb}
\usepackage{mathrsfs}

\newtheorem{theorem}{Theorem}
\newtheorem{corollary}{Corollary}
\newtheorem{lemma}{Lemma}[section]

\theoremstyle{definition}

\theoremstyle{remark}

\numberwithin{equation}{section}

\newcommand{\norml}[1]{\left\vert#1\right\vert}
\newcommand{\norm}[1]{\left\Vert#1\right\Vert}

\newcommand{\RR}{\mathbb{R}}
\newcommand{\NN}{\mathbb{N}}
\newcommand{\ZZ}{\mathbb{Z}}
\newcommand{\QQ}{\mathbb{Q}}
\newcommand{\mcD}{\mathcal{D}}
\newcommand{\mcC}{\mathcal{C}}

\DeclareMathOperator{\bad}{{\bf Bad}}

\DeclareMathOperator{\wdim}{windim}
\DeclareMathOperator{\bob}{\bf B}

\DeclareMathOperator{\alice}{\bf A}

\begin{document}
\title[The winning property of mixed badly approximable numbers]
{The winning property \\
of mixed badly approximable numbers}
\author{Yaqiao Li}
\address{School of Mathematical Sciences, Peking University, Beijing, 100871, China }
\email{yaqiaoli@pku.edu.cn}
\thanks{}
 \keywords{}

\begin{abstract}
For any pair of real numbers $(i,j)$ with $0<i,j<1$ and $i+j=1$, we
prove that the set of $p$-adic mixed $(i,j)$-badly approximable
numbers $\bad_p(i, j)$ is $1/2$-winning in the sense of Schmidt's
game. This improves a recent result of Badziahin, Levesley, and
Velani on mixed Schmidt conjecture.
\end{abstract}

\maketitle

\section{Introduction}\label{S:Intro}

Given a pair of real numbers $(i, j)$ such that
\begin{equation} \label{ij}
0 < i, j < 1 \qquad \text{ and } \qquad i+j=1.
\end{equation}
Let $\bad(i,j)$ denote the set of \emph{$(i,j)$-badly approximable
vectors} in $\RR^2$, that is,
\begin{align} \label{def:bad_ij}
\bad(i,j) :=\Big\{&(x, y) \in \RR^2 : \exists \  c(x,y)>0 \text{
such that } \notag \\
&\max\{ \norm{qx}^{1/i}, \norm{qy}^{1/j} \} > \frac{c(x,y)}{q},
\forall \ q \in \NN \Big\},
\end{align}
where $\norm{\cdot}$ denotes the distance of a number to its nearest
integer. The Schmidt conjecture says that for any two pairs of real
numbers $(i_1, j_1)$ and $(i_2, j_2)$ satisfying \eqref{ij}, we have
\[
\bad(i_1, j_1) \cap \bad(i_2, j_2) \neq \emptyset.
\]
Recently, Badziahin, Pollington, and Velani proved Schmidt
conjecture affirmatively in \cite{BPV} by showing that the
intersection of countably many $\bad(i_n, j_n)$ is of full Hausdorff
dimension. In two recent papers, An improved BPV's theorem by
showing that every $\bad(i,j)$ is a winning set in the sense of
Schmidt's game, see \cite{An1, An2}. Recall that any countable
intersection of $\alpha$-winning sets is also $\alpha$-winning, and
an $\alpha$-winning set is of full Hausdorff dimension, see
\cite{ScP, ScB}. Hence An indeed improved BPV's theorem.

We now consider the case in dimension one, the set
\eqref{def:bad_ij} is then reduced to the classical set of
\emph{badly approximable numbers}.
\begin{equation}\label{def:bad}
\bad := \Big\{x \in \RR : \exists \  c(x)>0 \text{ such that }
\norm{qx}
> \frac{c(x)}{q}, \forall \ q \in \NN \Big\}.
\end{equation}
Given $0<\alpha, \beta<1$, let $\gamma := 1 - 2\alpha +
\alpha\beta$, we say the pair $(\alpha, \beta)$ is \emph{admissible}
if $\gamma > 0$. A classical result of Schmidt says that $\bad$ is
$(\alpha, \beta)$-winning for all admissible $(\alpha, \beta)$.
Observe that $(\frac{1}{2}, \beta)$ is always admissible for every
$0 < \beta < 1$, so in particular, $\bad$ is $1/2$-winning.

In \cite{DT}, the $\mcD$-adic mixed Diophantine problems were
studied. Let $\mcD$ be a bounded sequence of positive integers
$(d_k)_{k=1}^\infty$, where every $d_k \geq 2$. Let $D_0 := 1$, $D_n
:= \prod_{k=1}^{n} d_k$. For every natural number $q \in \NN$,
define the $\mcD$-\emph{adic pseudo absolute value} as follows,
\[
\norml{q}_\mcD := \inf \Big\{ \frac{1}{D_n} : q \in D_n \ZZ \Big\}.
\]
The $\mcD$-adic pseudo absolute value reduces to the usual $p$-adic
norm if we let $\mcD$ be the constant sequence consisting of a prime
number $p$. Recently, Badziahin, Levesley, and Velani initiated the
study of mixed Schmidt conjecture in \cite{BLV}. Let
\begin{align}\label{def:bad_d_ij}
\bad_\mcD(i, j) := \Big\{& x\in \RR : \exists \  c(x)>0 \text{ such
that} \notag \\
& \max\big\{ \norml{q}_\mcD^{1/i}, \norm{qx}^{1/j} \big\}
> \frac{c(x)}{q}, \forall \ q \in \NN \Big\}.
\end{align}
We call this set as \emph{the set of mixed $(i, j)$-badly
approximable numbers}. The mixed Schmidt conjecture is then stated
as follows: for any two pairs of real numbers $(i_1, j_1)$ and
$(i_2, j_2)$ satisfying \eqref{ij}, we have
\[
\bad_\mcD(i_1, j_1) \cap \bad_\mcD(i_2, j_2) \neq \emptyset.
\]
In \cite{BLV}, they proved $\bad_\mcD(i, j)$ is $1/4$-winning, thus
resolved the mixed Schmidt conjecture affirmatively.

Now we recall the notion of winning dimension, which is introduced
in \cite{ScP}. The definition of $(\alpha,\beta)$-winning and
$\alpha$-winning will be reviewed in Section \ref{S:lemmas}. Let $S
\subset \RR^n$, the \emph{winning dimension} of $S$, denoted by
$\wdim S$, is defined as follows
\[
\wdim S = \sup \{0< \alpha < 1 : S \text{ is }\alpha\text{-winning}
\}.
\]
A result in \cite{ScP} says that if $S$ is a nontrivial subset, then
$0 \leq \wdim S \leq \frac{1}{2}$. For example, $\wdim \bad =
\frac{1}{2}$. BLV's theorem says that $\wdim \bad_\mcD(i, j) \geq
\frac{1}{4}$. Moshchevitin asked whether $\wdim \bad_\mcD(i, j) =
\frac{1}{2}$ in his recent survey on Diophantine problems, see
\cite{Mo}. This paper answers this question affirmatively. In fact,
our theorem is a natural generalization of Schmidt's classical
result on the set $\bad$.

\begin{theorem}\label{thm:mythm}
The set $\bad_\mcD(i, j)$ is $(\alpha, \beta)$-winning for all
admissible $(\alpha, \beta)$, in particular, $\bad_\mcD(i, j)$ is
$1/2$-winning.
\end{theorem}

As a result, $\wdim \bad_\mcD(i, j) = \frac{1}{2}$. A result in
\cite{ScB} says a set is $(\alpha, \beta)$-winning for $\gamma \leq
0$ if and only if this set is the whole set, accordingly, the
winning property of $\bad_\mcD(i, j)$ is the best possible, so we
give the best improvement of BLV's result in the sense of winning
dimension.

Given a prime number $p$, we use $\bad_p(i,j)$ to denote the set
defined by \eqref{def:bad_d_ij} where the $\mcD$-adic is replaced by
$p$-adic. Theorem \ref{thm:mythm} gives the following corollary.

\begin{corollary} \label{cor:mycorollary}
For any two different prime numbers $p, q$, for any two pairs of
real numbers $(i_1, j_1)$ and $(i_2, j_2)$ satisfying \eqref{ij},
the set $\bad_p(i_1, j_1) \cap \bad_q(i_2, j_2)$ is $1/2$-winning.
In particular,
\[
\bad_p(i_1, j_1) \cap \bad_q(i_2, j_2) \neq \emptyset.
\]
\end{corollary}

Note that the ``In particular'' part could also be deduced from
BLV's result.

The rest of the paper is organized as follows: in Section
\ref{S:lemmas}, we introduce the notion of Schmidt's game and
establish some notations, then we give two useful lemmas; the proof
of theorem \ref{thm:mythm} will be given in Section \ref{S:pf}.

\section{Schmidt's Game and Two Lemmas}\label{S:lemmas}

First we recall the notion of Schmidt's game, for details see
\cite{ScP, ScB}. In this paper we only consider Schmidt's game on
$\RR$, so we restrict our description only in this situation. Given
a set $S \subset \RR$, given two real numbers $0 < \alpha, \beta <
1$, two players, say Alice and Bob, will play the game.  The game is
played as follows, Bob starts the game by choosing a closed interval
$\bob_1 \subset \RR$, then Alice chooses an closed interval
$\alice_1$ such that $\alice_1 \subset \bob_1$ and $\rho(\alice_1) =
\alpha \rho(\bob_1)$, then Bob chooses another closed interval
$\bob_2$ such that $\bob_2 \subset \alice_1$ and $\rho(\bob_2) =
\beta\rho(\alice_1)$, then Alice chooses another closed interval
$\alice_2$ such that $\alice_2 \subset \bob_2$ and $\rho(\alice_2) =
\alpha \rho(\bob_2)$, and so on. Here $\rho(\bf A) = \frac{1}{2}
|\bf A|$, where $|\bf A|$ denotes the length of the interval $\bf
A$. We can see that intervals appearing in the game obey the
following relation, $\bob_1 \supset \alice_1 \supset \bob_2 \supset
\alice_2 \supset \ldots$. We say Alice wins the game if she can play
such that the single point in $\cap_{k=1}^\infty \alice_k =
\cap_{k=1}^\infty \bob_k$ lies in $S$, otherwise Bob wins. We say
$S$ is \emph{$(\alpha, \beta)$-winning} if Alice can always win the
game no matter how Bob plays, and $S$ is \emph{$\alpha$-winning} if
it is $(\alpha, \beta)$-winning for every $0 < \beta < 1$.

Let $\rho_k := \rho(\bob_k)$, then $\rho_{k+t} = (\alpha\beta)^t
\rho_k$. We now give the first lemma.

\begin{lemma} \label{lemma:schmidt}
Assume $(\alpha, \beta)$ is admissible. Let $t\in \NN$ be such that
$(\alpha\beta)^t < \frac{1}{2}\gamma$. Suppose an interval $\bob_k$
occurs in the $(\alpha, \beta)$ game, and suppose $y\in \RR$ is an
arbitrary fixed point, then Alice can play, no matter how Bob plays,
such that for every point $x \in \bob_{k+t}$,
\[
|x - y| > \frac{1}{2} \gamma \rho_k.
\]
\end{lemma}

This lemma is essentially due to Schmidt, we just write it in a
slightly different form in order to facilitate the proof of our
theorem. See Schmidt's book \cite{ScB} p.49 for a complete proof.
Here we give only the proof's main idea.

\begin{proof}
Without loss of generality, we could assume $y$ be the middle point
of $\bob_k$. Alice adopts the strategy that always selecting the
most left possible inscribed interval in each turn. Then after $t$
turns, all points in $\bob_{k+t}$ will satisfy the property in the
lemma.
\end{proof}

To give the next lemma we need some notations from \cite{BLV}, we
put them here for completeness. For any real number $c>0$, let
\[
\bad_\mcD(c; i, j) := \Big\{ x\in \RR : \max\{ \norml{q}_\mcD^{1/i},
\norm{qx}^{1/j} \}
> \frac{c}{q}, \forall \ q \in \NN \Big\},
\]
then we see
\[
\bad_\mcD(i, j) = \bigcup_{c>0} \bad_\mcD(c;i,j).
\]
Let
\[
\mcC_c := \left\{ \frac{r}{q} \in \QQ : (r, q) = 1, q>0, \text{ and
} \norml{q}_\mcD \leq \left( \frac{c}{q} \right)^{i} \right\}.
\]
Let $P = \frac{r}{q}$, and let
\[
\Delta_c(P) := \Big\{ x\in \RR : |x - P| \leq \frac{c^j}{q^{1+j}}
\Big\},
\]
then clearly we have
\[
\bad_\mcD(c;i,j) = \RR \backslash \bigcup_{P \in \mcC_c}
\Delta_c(P).
\]
Let $R \in \RR, R > 1$, let $t \in \NN$, both of which will be fixed
in Section \ref{S:pf}. Define
\[
\mcC_{c,k} := \Big\{P = \frac{r}{q} \in \mcC_c : R^{k-1} \leq
q^{\frac{1+j}{t}} < R^k \Big\},
\]
then we have
\[
\mcC_c = \bigcup_{k=1}^\infty \mcC_{c,k},
\]
hence
\begin{equation} \label{vital_relation}
\bad_\mcD(c;i,j) = \RR \backslash \bigcup_{k=1}^{\infty} \bigcup_{P
\in \mcC_{c,k}} \Delta_c(P).
\end{equation}
It is this relation that will be used in the proof in Section
\ref{S:pf}.

Now we give the following lemma, the idea of which is already in
\cite{BLV}.

\begin{lemma}\label{lemma:BLV}
For any two different points $P_s = \frac{r_s}{q_s} \in \mcC_{c,k},
s = 1, 2$, we have
\[
|P_1 - P_2| > c^{-i} R^{ \frac{i}{1+j}t(k-1) - \frac{1}{1+j}2tk}.
\]
\end{lemma}

\begin{proof}
By $P = \frac{r}{q} \in \mcC_{c,k} \subset \mcC_c$, we have
$\norml{q}_\mcD \leq (\frac{c}{q})^i$, by the definition of the norm
$\norml{q}_\mcD$, there is an appropriate $n\in \NN, q^{\ast} \in
\NN$ such that
\[
q = D_n q^{\ast}, \text{ and } q \notin D_{n+1} \ZZ,
\]
hence,
\[
D_n \geq \left(\frac{q}{c}\right)^i \geq c^{-i} R^{
\frac{i}{1+j}t(k-1)}.
\]
Now there will be $D_{n_1}$ and $D_{n_2}$ respectively for $P_1$ and
$P_2$, and one of them will divide another by the definition of
$D_n$, so $(q_1, q_2) \geq \min\{D_{n_1}, D_{n_2}\} \geq c^{-i} R^{
\frac{i}{1+j}t(k-1)}$. Therefore,
\[
|P_1 - P_2| \geq \frac{(q_1, q_2)}{q_1 q_2} >  c^{-i} R^{
\frac{i}{1+j}t(k-1) - \frac{1}{1+j}2tk}.
\]
\end{proof}

\section{Proof of theorem}\label{S:pf}
Now we prove theorem \ref{thm:mythm}. Given $(\alpha, \beta)$ be
admissible, then $0< \gamma < 1$. Fix one $t\in \NN$ such that
$(\alpha\beta)^t < \frac{1}{2}\gamma$, which is used in lemma
\ref{lemma:schmidt}. Our aim is to show that $\bad_\mcD(i, j)$ is
$(\alpha, \beta)$-winning. Without loss of generality we can assume
that $\rho_1$ is very small, so we take the following constants,
\begin{equation} \label{consts}
\begin{cases}
R = \frac{1}{\alpha\beta} > 1, \\
0 < \rho_1 < \left( \frac{1}{4} R^{-\frac{2t}{1+j}} \right)^j, \\
0 < c < (\frac{1}{2} \gamma \rho_1)^{1/j} < \rho_1^{1/j}.
\end{cases}
\end{equation}
As we pointed out in Section \ref{S:lemmas} that $\bad_\mcD(i, j) =
\bigcup_{c>0} \bad_\mcD(c; i,j)$, hence it suffices to show that for
the $c$ satisfying \eqref{consts}, $\bad_\mcD(c; i, j)$ is $(\alpha,
\beta)$-winning.

\begin{proof}
We prove it by showing the following two facts.

\textbf{Fact 1}. For every $k \geq 1$,
\[
\#\{ P \in \mcC_{c,k} : \Delta_c(P) \cap \bob_{t(k-1)+1} \neq
\emptyset \} \leq 1.
\]

\textbf{Fact 2}. Suppose \textbf{Fact 1} holds, then Alice can play,
no matter how Bob plays, such that for every $k \geq 1$,
\[
\#\{ P \in \mcC_{c,k} : \Delta_c(P) \cap \bob_{tk+1} \neq \emptyset
\} = 0,
\]
which is equivalent to
\[
\Delta_c(P) \cap \bob_{tk+1} = \emptyset, \quad \forall\  P \in
\mcC_{c,k}.
\]

Notice that the above equation implies
\[
\Delta_c(P) \cap \bob_{tk+1} = \emptyset, \quad \forall\  P \in
\mcC_{c,l}, \quad l = 1, 2, \ldots, k.
\]
Recall the relation \eqref{vital_relation}, then \textbf{Fact 2} is
equivalent to say that Alice can play such that the single point in
$\cap_{k=1}^\infty \bob_{tk+1}$ lies in $\bad_\mcD(c;i,j)$, so Alice
can always win the game and we are done. Hence we are only left to
show the two facts.

Now we show \textbf{Fact 1}. Let $z$ be the middle point of
$\bob_{t(k-1)+1}$. For those $P \in \mcC_{c,k}$ satisfying
$\Delta_c(P) \cap \bob_{t(k-1)+1} \neq \emptyset$, let $x \in
\Delta_c(P) \cap \bob_{t(k-1)+1}$, then
\[
|P - z| \leq |P - x| + |x - z| \leq \frac{c^j}{q^{1+j}} +
\rho_{t(k-1)+1} < 2\rho_1 R^{-t(k-1)}.
\]
Assume there are two points $P_1, P_2 \in \mcC_{c,k}$, $P_1 \neq
P_2$, and they satisfy $\Delta_c(P_1) \cap \bob_{t(k-1)+1} \neq
\emptyset$, $\Delta_c(P_2) \cap \bob_{t(k-1)+1} \neq \emptyset$. By
applying lemma \ref{lemma:BLV}, then
\[
c^{-i} R^{ \frac{i}{1+j}t(k-1) - \frac{1}{1+j}2tk} < |P_1 - P_2|
\leq |P_1 - z| + |z - P_2| < 4\rho_1 R^{-t(k-1)},
\]
which is equivalent to
\[
4\rho_1 c^i > R^{ (\frac{i}{1+j} + 1) t(k-1) - \frac{1}{1+j}2tk} =
R^{- \frac{2t}{1+j}}.
\]
Now use our assumption for $c$ in \eqref{consts}, then
\[
R^{- \frac{2t}{1+j}} < 4\rho_1 c^i < 4 \rho_1^{1 + \frac{i}{j}} = 4
\rho_1^{\frac{1}{j}},
\]
contradicts to our assumption on $\rho_1$.

Now we show \textbf{Fact 2}. We proceed by induction. The base is
quite clear. Suppose that an interval $\bob_{t(k-1)+1}$ occurs in
the game and it has empty intersection with all intervals
$\Delta_c(P)$ with $P \in \mcC_{c,k-1}$. By \textbf{Fact 1}, there
is not more than one ``dangerous'' point $P \in \mcC_{c,k}$ such
that $\Delta_c(P) \cap \bob_{t(k-1)+1} \neq \emptyset$. View this
point $P$ as the point $y$ in lemma \ref{lemma:schmidt}, by applying
that lemma, then Alice can play, no matter how Bob plays, such that
for all $x \in \Delta_c(P) \cap \bob_{tk+1}$, we have
\[
\frac{c^j}{q^{1+j}} \geq |x - P| > \frac{1}{2} \gamma
\rho_{t(k-1)+1} = \frac{1}{2} \gamma \rho_1 R^{-t(k-1)},
\]
since $P \in \mcC_{c,k}$, this gives
\[
c^j R^{-t(k-1)} \geq \frac{c^j}{q^{1+j}} > \frac{1}{2} \gamma \rho_1
R^{-t(k-1)},
\]
which reduces to
\[
c^j > \frac{1}{2} \gamma \rho_1,
\]
contradicts to our assumption on $c$. So $\Delta_c(P) \cap
\bob_{tk+1} = \emptyset$.
\end{proof}

\section*{Acknowledgments}
The author is grateful to Professor Jinpeng An for many valuable
conversations. He thanks the referee's advice on how to make the
proof more clear. The author also thanks Mengting Pan for her
selfless support and constant encouragement.

\end{document}